\documentclass[11pt]{article}

\usepackage{amsfonts, amssymb}
\usepackage{amsmath}
\usepackage{verbatim}
\usepackage{mathrsfs}
\usepackage{indentfirst}
\usepackage{color}
\usepackage{yhmath}     
\usepackage{color}

\begin{document}
\renewcommand{\theenumi}{(\roman{enumi})}
\newenvironment{proof}{\trivlist \item[\hskip \labelsep\hspace{5.5mm}{\bf Proof.}]}{
\endtrivlist}
\newtheorem{corollary}{\hspace{-2mm}Corollary}[section]
\renewcommand{\thecorollary}{\arabic{section}.\arabic{corollary}.}
\newtheorem{proposition}{\hspace{-2mm}Proposition}[section]
\renewcommand{\theproposition}{\arabic{section}.\arabic{proposition}.}
\newtheorem{theorem}{\hspace{-2mm}Theorem}[section]
\renewcommand{\thetheorem}{\arabic{section}.\arabic{theorem}.}
\newtheorem{lemma}{\hspace{-2mm}Lemma}[section]
\renewcommand{\thelemma}{\arabic{section}.\arabic{lemma}.}
\newtheorem{remark}{\hspace{-2mm}Remark}[section]
\renewcommand{\theremark}{\arabic{section}.\arabic{remark}.}
\newtheorem{defi}{\hspace{-2mm}Definition}[section]
\renewcommand{\thedefi}{\arabic{section}.\arabic{defi}.}
\renewcommand{\thesection}{\arabic{section}.}
\renewcommand{\thesubsection}{\arabic{section}.\arabic{subsection}.}
\newcounter{example}[section]
\newenvironment{example}{\refstepcounter{example} \trivlist
\item[\hskip \labelsep{\hspace{6mm}{\bf Example \theexample.}}]}{
\endtrivlist}
\newenvironment{claim}{\trivlist \item[\hskip \labelsep{\bf
Claim.} \it]}{ \endtrivlist}

\renewcommand{\labelenumi}{\normalfont (\roman{enumi})}
\makeatletter
\def\@begintheorem#1#2{\trivlist
  \item[\hskip 1cm {\bfseries #1\ #2}]\itshape}
\makeatother

\makeatletter
\def\enumerate{%
 \ifnum \@enumdepth >\thr@@\@toodeep\else
   \advance\@enumdepth\@ne
   \edef\@enumctr{enum\romannumeral\the\@enumdepth}%
     \expandafter
     \list
       \csname label\@enumctr\endcsname
       {\usecounter\@enumctr%
       \leftmargin 0pt%
       \itemindent 4mm 
       \def\makelabel##1{\hskip 11mm\llap{##1}}
       }%
 \fi}
\makeatother

\title{Orders of Finite Reductive Monoids
\footnote{Project 10471116 supported by NSFC} }
\date{}
\maketitle
\vspace{ -1.5cm}

\centerline {Zhuo Li  ~~Zhenheng Li$^1$ ~~You'an Cao}

\footnotetext[1]{Partially supported from 2006 summer stipend of
University of South Carolina Aiken, USA }


\vspace{ 0.2cm}

\def\J {{\cal J}}

\begin{abstract}
We show four formulas for calculating the orders of finite reductive monoids with zero. As applications, these formulas are then used to calculate the orders of finite reductive monoids induced from the $F_q$-split  $\J$-irreducible monoids $\overline {K^*\rho(G_0)}$ where $G_0$ is a simple algebraic group over the algebraic closure of $F_q$, and $\rho: G_0\to GL(V)$ is the irreducible representation associated with any dominant weight. Finally, we give an explicit formula for the orders of finite symplectic monoids associated with the last fundamental dominant weight of type $C_l$; the connections to $H$-polynomials and Betti numbers are shown.

\vspace{ 0.3cm}

\vspace{ 0.3cm} \noindent {\bf AMS Classifications:} 20G40; 20G15;
20M99

\end{abstract}

\def\a {\alpha}
\def\b {\beta}
\def\e {\varepsilon}
\def\s {\sigma}
\def\v {\vec}
\def\js {(\cal J, \sigma)}

\baselineskip 17pt


\section{Introduction}
Renner \cite{R3} extended the endomorphisms induced by the Frobenius maps and the graph automorphisms of the
Coxeter-Dynkin diagrams to reductive monoids. The monoids $M$ consisting of fixed points of the extended
endomorphisms are called {\it finite reductive monoids}. We are interested in formulas for calculating the
orders of these finite monoids.

Renner \cite{R4} established an enumerative theory of finite reductive monoids using a length function on Renner monoids.
Recently, in \cite{C2, R6, R7, R8} Can and Renner gave a systematic description of $H$-polynomials of reductive monoids,
whose whole point is to investigate the orders of finite reductive monoids.

Yan \cite{YAN} obtained a general formula for the orders of monoids of Lie type based on how the monoids are constructed.
He then used computers to calculate the order of $M_n(F_q)$ as an example. However, his formula can't be used to compute
the orders of finite reductive monoids in general because, with our notation, a key factor
$$
K(e) = \{g\in G \mid ge = e = eg\}
$$
in the formula was unknown except for $M_n(F_q)$, where $e$ is any element of the cross section lattice of $M$.
It turns out that $K(e)$ is critical in our work. A precise characterization of $K(e)$ is given in {Theorem 2.2}
in terms of the type map (see Definition 2.2) of $M$ and certain root subgroups of the unit group of $M$.

We show that, in Theorem 2.1, if ${\bold M}$ and ${\bold N}$ are reductive monoids and $F_q$-split in the sense
of \cite{H2} (see Section 34 of \cite{H2}), and there is a finite dominant morphism, defined over $F_q$, from
${\bold M}$ to ${\bold N}$, then $|M| = |N|$, where $M$ and $N$ are the finite reductive monoids consisting of
$F_q$-rational points of ${\bold M}$ and ${\bold N}$, respectively.

We then obtain four formulas in Section 3 for calculating the orders of the finite reductive monoids which are
fixed points of reductive monoids with zero under the standard Frobenius map. The first formula in {Theorem 3.1}
is based on the action of $G\times G$ on $M$ defined by
$
    (g, h)m=gmh^{-1},
$
where $m\in M$ and $(g,h)\in G\times G$. The structure of the isotropy group $(G\times G)_e$ and its cardinality
are described in {Proposition 3.1}, where $e\in\Lambda$. The second formula given in Theorem 3.2 is a refinement
of the first by using the structure of $K(e)$. A more theoretically useful formula, the third, is shown in Theorem 3.3
in terms of polynomials $\sum_{w\in D(e)}q^{l(w)}$ and $\sum_{w\in
D_{*}(e)}q^{l(w)}$, where $D(e)$ and $D_*(e)$ are appropriate
subsets of the Weyl group. The fourth formula, which is more practically useful,
is characterized in {Theorem 3.4} using the degrees of certain
basic polynomial invariants of the parabolic subgroups of $W$.

The above formulas compute the orders of finite reductive monoids with zero when their type maps
are known. So far, we know explicitly the type maps of the $\J$-irreducible monoids (\cite{PR2}),
the 2-reducible monoids (\cite{R1}), and the multilined closure monoids (\cite{LP1, LR1}).
In {Theorem 4.1} we determine the orders of the finite reductive monoids induced from the
$F_q$-split $\J$-irreducible monoids $\overline {K^*\rho(G_0)}$ where $G_0$ is a simple algebraic group
over the algebraic closure of $F_q$, and $\rho: G_0\to GL(V)$ is the irreducible representation associated
with any dominant weight.

Finally, we explicitly compute in {Theorem 4.2} the orders of finite symplectic monoids induced from
$\overline {K^*\rho(G_0)}$ associated with the last fundamental dominant weight of type $C_l$, and then
show connections with $H$-polynomials in Proposition 4.2 and Betti numbers in Corollary 4.1.


\section{Finite Reductive Monoids}

An {\it algebraic monoid} is an affine variety defined over an
algebraically closed field $K$ together with an associative
morphism and an identity. The unit group of an algebraic monoid is
an algebraic group. An algebraic monoid is {\it irreducible} if it
is irreducible as a variety. An irreducible monoid is called {\it
reductive} if its unit group is a reductive group. A reductive
monoid with zero is $\J$-{\it irreducible} if its cross section
lattice has a unique minimal non-zero idempotent. A systematic description of the theory of reductive
monoids can be found in Putcha \cite{PU1}, Renner \cite{R1}, and  Solomon \cite{LS1}.

We now explain the construction of finite reductive monoids from reductive monoids as found in \cite{R3}.
From now on, let $F_q$ be a finite
field with $q$ elements and $K$ the algebraic closure of $F_q$.
Suppose that ${\bold M}$, with 0 and unit group ${\bold G}$, is a reductive monoid over $K$.
Let $\sigma: {\bold M} \to {\bold M}$
be an endomorphism such that $\sigma$ is a finite morphism and
$G = \{x\in {\bold G} \mid \sigma(x) = x\}$ is a finite group.

\begin{defi} Let ${\bold M}$ be reductive with $\sigma$ as above.
By a finite reductive monoid $M$, we mean
\[
M = \{x\in {\bold M} \mid \sigma(x) = x\}.
\]
\end{defi}

If ${\bold M} = M_n(K)$ consisting of all $n\times n$
matrices over $K$ and $\sigma$ be the Frobenius map from $M_n(K)$ to
$M_n(K)$ defined by: $[a_{ij}]\mapsto [a_{ij}^q]$, then $M =
M_n(F_q)$.

We next consider relationships among the orders of finite reductive monoids induced from reductive monoids
${\bold M}$ defined over $F_q$. We say that ${\bold M}$ is $F_q$-split
if its unit group ${\bold G}$ is $F_q$-split in the sense of \cite{H2}. Let $M$ be the monoid of $F_q$-rational
points of ${\bold M}$ (c.f. Section 34 of \cite{H2}). It follows
from Section 4 of \cite{R5} that there exists an $F_q$-automorphism
$\sigma$: ${\bold M} \to {\bold M}$ of algebraic monoids such that
$M = \{x\in {\bold M} \mid \sigma(x) = x\}$.
Thus, $M$ is a finite reductive monoid. More properties
about the finite reductive monoids can be found in \cite{PR3, R3, R4}.
Renner pointed out the following theorem to the authors.

\begin{theorem} Let ${\bold M}$ and ${\bold N}$ be $F_q$-split reductive monoids, and let $M$ and $N$ be two finite
monoids of $F_q$-rational points of ${\bold M}$ and ${\bold N}$, respectively. If $h : {\bold M} \rightarrow {\bold N}$
is a finite dominant morphism defined over $F_q$, then $|M| = |N|$.
\end{theorem}
\begin{proof}
Let ${\bold B}$ be a Borel subgroup of the unit group of ${\bold M}$. Then $h({\bold B})$ is a Borel subgroup of
${\bold N}$.  Note that $h$ induces an isomorphism on the Renner monoids of ${\bold M}$ and ${\bold N}$.  Denote by
${\bold R}$ the Renner monoid
of ${\bold M}$. If ${\bold M}$ is the disjoint union of ${\bold B}r{\bold B}$ for $r\in {\bold R}$, then ${\bold N}$
is the disjoint union of $h({\bold B})h(r)h({\bold B})$ for $r\in {\bold R}$. Since ${\bold B}$ and $h({\bold B})$
are connected solvable groups, for $r\ne 0$, ${\bold B}r{\bold B}$ is isomorphic, as a variety, to $K^a \times (K^*)^b$
for integers
$a\ge 0$ and $b>0$.
But so is $h({\bold B})h(r)h({\bold B})$ because $h$ is finite and dominant. Also remember that everything is
$F_q$-split
here. So $|M|$ is a sum of the $F_q^a\times (F_q^*)^b = q^a(q-1)^b$. The same sum works for $|N|$. $\hfill\Box$
\end{proof}

If  $\tilde{{\bold M}}$ is the normalization of ${\bold M}$ then $\tilde{{\bold M}}$ is a reductive monoid.
Let $\tilde{M}$ be the set of $F_q$-rational points of $\tilde{{\bold M}}$. The above theorem shows that
$|\tilde{M}| = |M|$.


\subsection{Structure of K(e) for Finite Reductive Monoids}
Let $G$ be the group of units of a finite reductive monoid $M$. Suppose that $B$ and
$B^-$ are opposite Borel subgroups and $T\subseteq B$ a maximal
torus. Denote by $W=N_G(T)/T$ the Weyl group. For $X\subseteq M$,
we define $E(X)=\{e\in X \mid e^2=e\}$ to be the set of
idempotents in $X$. For $e\in E(M)$, let $P(e) = \{g\in G \mid
ge=ege\}$ and $P^{-}(e) = \{g\in G \mid eg=ege\}$. Then
$\Lambda=\{e\in E(M) \mid B \subseteq P(e), B^-\subseteq P^-(e)\}$
is called a cross section lattice of $M$. We have a decomposition of $M$ into $G\times G$ orbits
\[
    M=\bigsqcup_{e\in \Lambda}GeG.
\]
The monoid $R = \langle \, W, ~\Lambda \, \rangle$ is referred to as the Renner monoid of $M$
(see \cite{PU2, PU3, PU4, R4}).
Then $M$ has a Bruhat-Renner decomposition
\[
    M=\bigsqcup_{x\in R} BxB.
\]
The following results are well known for finite reductive monoids with zero \cite{PR3, R3, R4}.

\vspace{-2mm}
\begin{proposition}{\label{BN}}
    \begin{enumerate}
    \item For $e\in E(M)$, the above $P(e)$ and $P^{-}(e)$ are
    opposite parabolic subgroups of $G$, and
    $
        U(e)e = \{e\} = eU^-(e),
    $
    where $U(e)$ and $U^-(e)$ are the unipotent radicals of $P(e)$ and $P^-(e)$, respectively.
    \vspace{-1mm}
    \item For any $e, f\in E(M)$, $eM=fM $ or $Me=Mf$ implies that $e = x^{-1}fx$ for some $x\in G$.
    \end{enumerate}
\end{proposition}

\vspace{-3mm}
Putcha \cite{PU3} calls an abstract monoid $M$ with unit group $G$
{\it a monoid of Lie type} if $M=E(M)G$ satisfies (i) and (ii) of
the above theorem and $G$ is a finite group of Lie type. It
follows from Section 4 of \cite{R3} that finite reductive monoids
are monoids of Lie type; see also \cite{PU2, PU3}.

\begin{defi} Let $M$ be a finite reductive monoid and $\Delta$ be the set of simple roots.
The type map $\lambda: \Lambda
\rightarrow 2^\Delta$ is given by
$$
\lambda(e) = \{\alpha \in \Delta \mid s_\alpha e = e s_\alpha \}.
$$
\end{defi}

Let
$
    \lambda^*(e) = \{ \alpha \in \Delta \mid s_\alpha e = e s_\alpha \ne e \}
    \text{ and } \lambda_*(e) =  \{ \alpha \in \Delta \mid s_\alpha e = e s_\alpha = e \}.
$
Then we have the following parabolic subgroups of the Weyl group $W$.
\[
\begin{aligned}
    W(e)  &= W_{\lambda(e)} = \{ w \in W \mid w e = e w \}.         \\
    W_*(e)&= W_{\lambda_*(e)} = \{ w \in W \mid w e = e w = e \}.   \\
    W^*(e)&= W_{\lambda^*(e)}.      \\
    W(e)  &= W_*(e) \times W^*(e).
\end{aligned}
\]
For $e\in\Lambda$, denote by $\Phi_{\lambda_*(e)}\subseteq \Phi$ the root subsystem with a
base $\lambda_*(e)\subseteq \Delta$. Let $U_{\alpha}$ be the root subgroup determined by
$\a \in \Phi_{\lambda_*(e)}$. Construct a subgroup of $G$ as
follows
$$
G_{\lambda_*(e)} =\langle \, T(e), U_{\a} \mid \a \in \Phi_{\lambda_*(e)}  \, \rangle
$$
where
$
    T(e) = \{t\in T \mid te = e = et \}.
$
Thus $G_{\lambda_*(e)}$ is determined by $T(e)$, the type map of
$M$ and certain root subgroups of $G$.

The subgroup of $M$ below plays an important role in our work
$$
    K(e) = \{g\in G \mid ge = e = eg\}.
$$
Putcha \cite{PU3} gave a precise description of ${ K}(e)$ for the universally maximal monoid in the context of
monoids of Lie type. Then Putcha and Renner gave another characterization of ${ K}(e)$ in \cite{PR4}.

We describe $K(e)$ in a different way by claiming that $K(e) = G_{\lambda_*(e)}$. If $e=1$, then
${K}(e) = {G}_{\lambda_*(e)} =
\{1\}$. A simple calculation yields that, if $e = 0$, then ${K}(e) = {G}_{\lambda_*(e)} = {G}$.
By Lemma 7.4 (a) of \cite{R1}, for any $\a\in \lambda_*(e)$, we have ${U}_\a e = e {U}_\a = e$,
for $e\in {\Lambda} \setminus
\{0, 1\}$. Thus ${G}_{\lambda_*(e)} \subseteq {K}(e)$,
for all $e\in {\Lambda}$. To prove $K(e) = G_{\lambda_*(e)}$ in general,
we need the concept of the canonical (or normal) form of an element in $G$, which can be found in
Section 28.4 of \cite{H2} or Section 8.4 of \cite {C1}. Let $G, B, T$, and $W$ be as above, and denote
by $U$ the unipotent radical of $B$.

\begin{lemma}\label{lemma 2.1}
For an arbitrary $g\in {G}$, let $g=utwv$ be its canonical form, where $u, v\in {U}$, $t\in {T}$ and $w\in {W}$.
Then $g\in {K}(e)$ if and only if $u,v, w, t\in {K}(e)$, for $e \in {\Lambda}$.
\end{lemma}

\begin{proof}
It is obvious that if $u,v, w, t\in {K}(e)$ then $g=utwv\in
{K}(e)$, where $e \in {\Lambda}$. Now, we assume
$g=utwv\in {K}(e)$, in other words, $utwve=eutwv=e$. Clearly, $utwv\in
{L}(e)={P}(e)\cap {P}^-(e)$. Hence $u, t,w, v\in
{L}(e)$. Then $we = ew$ and $ve = ev$. It follows from
$utwve=e$ that $utwev=e$, that is, ${B}we{B}={
B}e{B}$. Thus $we=e=ew$, and so $w\in {K}(e)$ and $w\in W_{\lambda_*(e)}$. It
follows easily that $utev = e$. Equivalently,
\[
te = u^{-1}ev^{-1}. \tag{1}
\]

To see that $t\in {K}(e)$, we first prove that $te$ is an idempotent. By Theorem 8.6 of
\cite{H2}, we regard ${G}$ as a closed subgroup of some
$GL(n, F_q)$. Then we can choose ${T}, {B}, {U}$,
up to conjugation, such that ${T}$ consists of invertible
diagonal matrices, that ${B}$ is composed of invertible upper
triangular matrices, and that ${U}$ is made of invertible
upper triangular matrices with all diagonal entries 1. The left
hand side of (1) is a diagonal matrix and the right hand of (1) is
upper triangular with diagonal entries 1 or 0. So $te$ is an idempotent. In other words, $tete =
te$. It follows that $ete = e$, and hence $te = e$. Therefore, $t\in {K}(e)$.

To show that $v\in K(e)$, note that $g=utwv$ is in its canonical
form. This tells us that $v\in \Pi_{\a\in {\Psi}} {
U}_{\a} \subseteq {P}(e)$ where ${\Psi} = \{\a \in
{\Phi}_{\lambda(e)}^+ \mid w(\alpha) \in {
\Phi}_{\lambda(e)}^-\}$. Notice that $w\in W_{\lambda_*(e)}$ and that $\lambda(e)$ is a disjoint union of
$\lambda^*(e)$ and $\lambda_*(e)$. It follows that ${\Psi}\subseteq
{\Phi}_{\lambda_*(e)}$. Hence $ve=ev=v$, which shows that $v\in K(e)$.

Finally, $ue=eu=e$, which means that $u\in {K}(e)$.
\hspace{1cm}$\Box$
\end{proof}

\begin{theorem}\label{theorem 2.2} For $e\in
\Lambda$, we have $ K(e) = G_{\lambda_*(e)}$.
\end{theorem}

\begin{proof}
It suffices to show ${
G}_{\lambda_*(e)} \supseteq {K}(e)$. For any $g = uwtv \in
{K}(e)$ in its canonical form, it follows from Lemma
\ref{lemma 2.1} that $u, w, t, v \in {K}(e)$. It is clearly
true that $w, t, v \in {G}_{\lambda_*(e)}$ by the proof of
Lemma \ref{lemma 2.1} We now show that $u\in {
G}_{\lambda_*(e)}$. Let ${U}_e = \{x\in {U} \mid xe =
ex = e\}$. Then ${U}_e$ is a closed, ${T}$-stable
subgroup of ${U}$. It follows from Proposition 28.1 of
\cite{H2} that the subgroup ${U}_e\subseteq {
G}_{\lambda_*(e)}$ is generated by certain root subgroups. Since
$u\in {U}_e$, we can then assume $u\in \prod_{i} {
U}_{\beta_i}$ where ${U}_{\beta_i}e = e{
U}_{\beta_i}=e$ and $\beta_i\in {\Phi}$. By Exercise 2 of
7.7.2 in \cite{R1}, we see that $\beta_i\in {\Phi}_{\lambda_*(e)}$.
This completes the proof. \hspace{1cm}$\Box$
\end{proof}


\section{Orders of Finite Reductive Monoids}
Define an action of $G\times G$ on $M$ by:
\[
    (g, h)m = g m h^{-1},
\]
where $m\in M$ and $(g, h)\in G\times G$. The structure of the isotropic group $(G\times G)_e$ for $e\in
\Lambda$ plays an important role
in the procedure of obtaining formulas for the orders of the
finite reductive monoids. The following result describes the
structure of $(G\times G)_e$ and its cardinality $|(G\times
G)_e|$, where $e\in \Lambda$.
\begin{proposition} \label{orbit} Let $e\in \Lambda$. Then
\begin{enumerate}
\item $(G\times G)_e   = \{ (lu, lkv)\in G\times G \mid u\in U(e), v\in U^-(e), l\in L(e) \mbox{ and } k\in K(e) \}$.
\vspace{-2mm}
\item $|(G\times G)_e| = |P(e)||U(e) ||K(e)|.$
\end{enumerate}
\end{proposition}

\begin{proof} Assume that $geh^{-1}=e$ for $e\in \Lambda$ and $(g,h)\in G\times G$.
Then $egeh^{-1}=ee=e=geh^{-1}$, and hence $ege=ge$, which means
$g\in P(e)$. Similarly, $h\in P^-(e)$. Let $P(e)=U(e) L(e) $ and
$P^-(e) =U^-(e)L(e) $ be the Levi decompositions with $L(e)
=P(e)\cap P^-(e)$. We have $g=l_1u$ and $h=l_2v$ for $u\in U(e) $,
$v\in U^-(e)  $, and $l_1, l_2\in L(e) $. It follows from (i) of
Proposition \ref{BN} that
\[
geh^{-1}=l_1uev^{-1}l_2^{-1}=el_1l_2^{-1}=l_1l_2^{-1}e=e,
\]
and hence $l_1l_2^{-1}\in K(e)$. This proves (i). It not
difficult to see that (ii) is correct by (i) and $|U(e)| =
|U^-(e)|$. \hspace{1cm}$\Box$
\end{proof}


\begin{theorem}\label{formula}\label{f1} Let $M$ be a finite reductive monoid with
zero and unit group $G$. Then
\[
  |M| = \sum_{e\in \Lambda}\frac{|G|^2}{|P(e)||K(e)||U(e) |}
      = \sum_{e\in \Lambda} [ G : P(e)]^2[L(e) : K(e) ]. \\
\]
\end{theorem}

\begin{proof} Since $ M = \bigsqcup_{e\in \Lambda} GeG$ and $|GeG| = |G|^2 / |(G\times G)_e|$, the first
equality follows from Proposition \ref{orbit} The second equality follows from that $|P(e)| = |L(e)||U(e)|$.
$\hfill\Box$
\end{proof}

{\bf Remark:} For an abstract monoid of Lie type, Yan found the
above formula in his Ph.~D. thesis in 1996 (see also \cite{YAN})
in a different way. He illustrated this formula by checking the
order of $M_n(F_q)$.


The result below follows immediately from Theorems \ref{theorem 2.2} and \ref{formula}

\begin{theorem}\label{f2}  The order of a finite reductive
monoid $M$ with zero is given by
\[
    |M| = \sum_{e \in \Lambda}\frac{|G|^2}{|P(e)||G_{\lambda_*(e)}||U(e)|}.
\]
\end{theorem}

We further refine the formulas in Theorems \ref{f1} and \ref{f2}
Let $B = UT$ be the Levi decomposition with the Levi factor $T$
and the unipotent radical $U$ of $B$. Let $J \subseteq\Delta$ and
$P_J$ a standard parabolic subgroup of $G$. Define $W_J$ to be the
parabolic subgroup of $W$ generated by reflections determined by
simple roots in $J$. For any $w\in W$, let $l(w)$ denote the length of $w$.

\begin{lemma}\label {lemma 3.1}
\begin{enumerate}
\item $|G| = |U||T|\sum_{w\in W} q^{l(w)}$.
\item $|P_J| = |U||T| \sum_{w\in W_J} q^{l(w)}$, for any $J \subset
\Delta$.
\end{enumerate}
\end{lemma}

\begin{proof}
The idea of Section 8.6 of \cite {C1} applies here proving (i). As for (ii), the canonical form of elements in $G$
indicates that each element of $P_J$ has a unique expression in
the form $bwu$, where $b\in B, w\in W_J$, and $u\in U_{w}^{-} =
\prod\limits_{\a\in \Psi}U_{\a}$ with $\Psi = \{\a \in \Phi^+ \mid
w(\a) \in \Phi^-\}$. Then
$$
|P_J| = |U||T|\sum_{w\in W_J}|U_{w}^{-}|.
$$
But then $|U_w^-| = q^{l(w)}$. Therefore, (ii) is proved.
\hspace{1cm}$\Box$
\end{proof}


\begin{theorem}\label {theorem DDE} For $e\in \Lambda$,
let $D(e)$ be the set of minimal length left coset representatives
of $W(e)$ in $W$, and $D_*(e)$ the set of minimal length left
coset representatives of $W_*(e)$ in $W$. Then
$$
|M| =  \sum_{ e\in \Lambda } {\Big(}[T:T(e)]
~q^{N^*(e)}\sum_{w\in D(e)} q^{l(w)} \sum_{w\in D_*(e)} q^{l(w)}
{\Big )},
$$
where $T(e) = \{t\in T \mid te = e = et \}$ is a maximal torus of
$G_{\lambda_*(e)}$ and $N^*(e) = |\Phi_{\lambda^*(e)}^+|$, the
number of positive roots in $\Phi_{\lambda^*(e)}$.
\end{theorem}

\begin{proof}

To see this is true, we use Theorem \ref{formula} First, $|G| =
|U||T|\sum_{w\in W} q^{l(w)}$ by (i) of Lemma \ref{lemma 3.1}
Then, consider parabolic subgroups $P(e) = BW(e)B$ of $G$. It
follows from (ii) of Lemma \ref{lemma 3.1} that $ |P(e)| = |U||T|
\sum_{w\in W(e)} q^{l(w)}, $ where $T\subseteq B$ is the maximal
torus of $G$, and $U$ the unipotent radical of $B$. So
\[ [G:P(e)]
=\frac{ \sum_{w\in W} q^{l(w)} }  { \sum_{w\in W(e)} q^{l(w)}
}.\tag{1}
\]
Let $U(e)$ be the unipotent group of $P(e)$. Then $ U(e) = \prod
\limits_{\a\in\Phi^+ \setminus \Phi_{\lambda(e)}} U_{\a}$ by a
result on page 119 of \cite{C1}, and hence $|U(e)| =
q^{N - N(e)}$ where $N$ and $N(e)$ are the numbers of positive
roots in $\Phi$ and $\Phi_{\lambda(e)}$, respectively. By the Levi
decomposition of $P(e) = L(e)U(e)$, we see that
$$
|L(e)| = q^{N(e)}|T| \sum_{w\in W(e)} q^{l(w)}.
$$
From Theorem \ref{theorem 2.2} and Lemma \ref{lemma 3.1}, $
|K(e)|= |G_{\lambda_*(e)}| = q^{N_*(e)} |T(e)| \sum_{w\in W_*(e)}
q^{l(w)}, $ where $N_*(e)= |\Phi^+_{\lambda_*(e)}|$. We then
obtain
\[
[L(e):K(e)]  = [T:T(e)] ~q^{N(e) - N_*(e)}
 \sum_{w\in W(e)} q^{l(w)} \Big/
\sum_{w\in W_*(e)} q^{l(w)}. \tag{2}
\]
It follows from (1) and (2) that
$$
\begin{aligned}
{[G:P(e)]^2[L(e):K(e)]}  &= [T:T(e)] ~q^{N(e) - N_*(e)} \frac{
(\sum_{w\in W} q^{l(w)})^2 } { \sum_{w\in W(e)} q^{l(w)}
\sum_{w\in W_*(e)} q^{l(w)} },\\
                        &= [T:T(e)] ~q^{N^*(e)} \sum_{w\in D(e)}
q^{l(w)} \sum_{w\in D_*(e)} q^{l(w)},
\end{aligned}
$$
where $N^*(e) = N(e) - N_*(e)$ is the number of positive roots in
$\Phi_{\lambda^*(e)}$, and $D(e)$ and $D_*(e)$ are the sets of
minimal length left coset representatives of $W(e)$ and $W_*(e)$
in $W$, respectively (more information on $D(e)$ and $D_*(e)$ can
be found in \cite{PU5, R1}). Using Theorem \ref{formula}, we
complete the proof of Theorem \ref{theorem DDE}
\hspace{1cm}$\Box$
\end{proof}

Although this is a useful theoretical formula for the order of
$M$, it is not easy to use in practice because
\[
\sum_{w\in D(e)} q^{l(w)} \sum_{w\in D_*(e)} q^{l(w)}
\]
is very cumbersome. We will give a simplified version of the
formula in Theorem \ref{theorem fTE} below.

For any $e\in\Lambda$, suppose that $\lambda_*(e)$
has $s$ connected components: $J_1, J_2, ..., J_s$, and that
$\lambda^*(e)$ has $t-s$ connected components: $J_{s+1}, , \cdots,
J_t$ with $|J_k| = m_k$ for $k=1, ..., t$. The degrees $d_1, d_2,
\cdots, d_{m_k}$ of the basic polynomial invariants of $W_{J_k}$
depend on the type of $J_k$ (one of the $A_l, B_l, C_l, D_l, E_6,
E_7, E_8, F_4, G_2$). To emphasize this dependence and to agree to
the notation in Chapter 9 of \cite {C1}, for $k=1, ..., t$, let
\[
P_{W_{J_k}}(q) = \prod_{i=1}^{m_k} \Big(\frac
{q^{d_i}-1}{q-1}\Big).
\]
Clearly, $P_{W}(q) = \prod\limits_{i=1}^{l}\Big(\frac
{q^{d_i}-1}{q-1}\Big)$. We can now prove the following more
practical formula.


\begin{theorem} \label{theorem fTE}
The order of a finite reductive monoid $M$ with 0 is
\[
|M| =  \sum_{ e\in \Lambda } \frac { [T:T(e)]
~q^{N^*(e)} P_W^2(q) } { \prod \limits_{k=1}^{s} P_{W_{J_k}}^2(q)
\prod\limits_{k=s+1}^{t} P_{W_{J_k}}(q) } .
\]

\end{theorem}
\begin{proof}
A theorem of Solomon, Theorem 9.4.9 of \cite {C1}, tells us that $\sum\limits_{w\in W} q^{l(w)} = P_W(q),$
\[
\begin{aligned}
\sum_{w\in W_*(e)} q^{l(w)} &= P_{W_*(e)}(q) = \prod \limits_{k=1}^{s} P_{W_{J_k}}(q), \mbox { ~and }\\
\sum_{w\in W(e)} q^{l(w)}   &= P_{W(e)}(q) = \prod \limits_{k=1}^{t} P_{W_{J_k}}(q) = \prod \limits_{k=1}^{s} P_{W_{J_k}}(q) \prod \limits_{k=s+1}^{t} P_{W_{J_k}}(q).\\
\end{aligned}
\]
Theorem \ref {theorem fTE} follows from these identities and Theorem \ref {theorem DDE} $\hfill\Box$
\end{proof}

\section{Applications}

In this section we show applications of our results by finding an explicit formula for the orders of $F_q$-split finite  $\J$-irreducible monoids, and then illustrate the formula by giving a precise description of the orders of symplectic finite reductive monoids. We then describe the connections between these orders and the H-polynomials as well as Betti numbers of the projective varieties of symplectic monoids.

\subsection{Orders of $F_q$-split Finite $\J$-irreducible Monoids}
Let $G_0$ be a simple algebraic group of type $X_l$, where $X_l = A_l, B_l, C_l, D_l, E_6, E_7, E_8, F_4$, or $G_2$, and $\rho: G_0\to GL(V)$ be an irreducible representation associated with any dominant weight. Then ${\bold M} = \overline {K^*\rho(G_0)}$ is a $\J$-irreducible monoid.  Suppose that $\sigma$ is the standard Frobenius map of ${\bold M}$. Then $M = \{x\in {\bold M} \mid \sigma(x)=x\}$ is a finite reductive monoid with 0, and it is called a finite $\J$-irreducible monoid. Let $T$ be a maximal torus of the unit group $G$ of $M$, and let $\Lambda$ be the cross section lattice of $M$. Recall that $T(e) = \{t\in T \mid te = e = et \}$, for $e\in \Lambda$. In this section we assume that ${\bold M}$ is $F_q$-split.

\begin{lemma}\label{lemma TTE} For $e\in
\Lambda\setminus \{0\}$,
$$[T:T(e)] =(q-1)^{|{\lambda^*(e)}| + 1}.$$
\end{lemma}
\begin{proof} Note that $|T(e)| =(q-1)^{\text{dim }T - \text{dim }eT}$, since ${\bold M}$
is $F_q$-split. By Theorem 4.16 of \cite{PR2}, $\text{dim
}eT=|\lambda^*(e)| + 1$. This proves the lemma. \hspace{1cm}$\Box$
\end{proof}

\begin{theorem} \label{theorem OJS} Let $M = \{x\in {\bold M} \mid \sigma(x)=x\}$ be a finite reductive monoid as above. Then
$$
|M| =  \sum_{ e\in \Lambda } \frac { q^{N^*(e)}
(q-1)^{2(|\lambda(e)|-l)+1} \prod\limits_{i=1}^{l} (q^{d_i} - 1)^2
} { \prod \limits_{k=1}^{s} \prod\limits_{i=1}^{m_k} (q^{d_i} -
1)^2 \prod\limits_{k=s+1}^{t} \prod\limits_{j=1}^{m_k} (q^{d_j} -
1)  }.
$$
\end{theorem}

\begin{proof} By Lemma \ref{lemma TTE} we obtain $[T:T(e)] = (q-1)^{|\lambda^*(e)|+1}$, for $e \in \Lambda\setminus\{0, 1\}$. Note also that, $P_W^2(q) =
(q-1)^{-2l}\prod\limits_{i=1}^{l} (q^{d_i} - 1)^2$ and
$$
\prod \limits_{k=1}^{s} P_{W_{J_k}}^2(q) \prod\limits_{k=s+1}^{t} P_{W_{J_k}}(q)
= \frac { \prod \limits_{k=1}^{s} \prod\limits_{i=1}^{m_k}
(q^{d_i} - 1)^2 \prod\limits_{k=s+1}^{t} \prod\limits_{j=1}^{m_k}
(q^{d_j} - 1) } { (q-1)^{|\lambda(e)| + |\lambda_*(e)|} }.
$$
It follows from Theorem \ref{theorem fTE} that this theorem holds. \hspace{1cm}$\Box$
\end{proof}

The above theorem tells us that the degrees $(d_1, d_2, ..., d_l)$ of the basic
polynomial invariants of W are critical in calculating the orders
of finite reductive monoids. They are described in standard books, for example, Proposition 10.2.5 of \cite {C1}.
For convenience, we list them.

\begin{proposition} \label{prop 5.2}
The degrees $(d_1, d_2, ..., d_l)$ of the basic polynomial invariants of W are given by

\hspace{1cm} $A_l: (2, 3, \cdots, l, l+1)$, for  $l\ge 1$.

\hspace{1cm} $B_l: (2, 4, \cdots, 2(l-1), 2l)$, for $l\ge 2$.

\hspace{1cm} $C_l: (2, 4, \cdots, 2(l-1), 2l)$, for $l\ge 3$.

\hspace{1cm} $D_l: (2, 4, \cdots, 2(l-1), l)$, for $l\ge 4$.

\hspace{1cm} $E_6: (2, 5, 6, 8, 9, 12)$.

\hspace{1cm} $E_7: (2, 6, 8, 10, 12, 14, 18)$.

\hspace{1cm} $E_8: (2, 8, 12, 14, 18, 20,  24, 30)$.

\hspace{1cm} $F_4: (2, 6, 8, 12)$.

\hspace{1cm} $G_2: (2, 6)$.
\end{proposition}

For $r=1, ..., n$ let
\[
[n, r, q] = \frac{(q^{n} - 1) \cdots (q - 1)}{[(q^{r} - 1) \cdots (q - 1)] [(q^{n-r} - 1) \cdots (q - 1)]}.
\]
By convention, if $r=0$, then let $[n, r, q] = 1$. For the monoid $M_n(F_q)$ of type $A_l$ with $n = l + 1$,
apply Theorem \ref{theorem OJS} to $M_n(F_q)$. We have
\[
    |M_n(F_q)| - 1 = \sum_{r=1}^{n}
                            q^{r(r-1)/2}
                            [n, r, q]^2
                            \prod_{i=1}^{r}(q^{i} - 1).
\]
Furthermore, let $M^r$ be the set of matrices in $M_n(F_q)$ of rank $r$, for $r=0, ..., n$. Then
\[
    |M^r| = q^{r(r-1)/2}[n, r, q]^2 \prod_{i=1}^{r}(q^{i} - 1).
\]
These formulas are recorded in Solomon \cite{LS0} and Renner \cite{R4}. In the next section, we calculate the orders of finite symplectic monoids.

\subsection{Orders of Finite Symplectic Monoids}
Let $n= 2l$ be an even positive integer, and $Sp_n$ be the symplectic group over the algebraic closure of $F_q$ (see Humphreys \cite{H2}). The Dynkin diagram of $Sp_n$ is as follows.

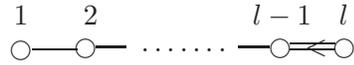
\setlength{\unitlength}{1mm}
\vskip 10mm
\begin{center}
\begin{picture}(80,15)

\put (15,15.5) {\circle{2.5}}

\put (23.6,15.8) {\circle{2.5}}

\put (31, 15.4) {.\,.\,.\,.\,.\,.\,.}

\put (49.5,15.8) {\circle{2.5}}

\put (52.8,14.92) {$<$}

\put (58,15.8) {\circle{2.5}}

\put (14.1,19) {$1$}

\put (23.3,19) {$2$}

\put (45.8,19) {$l-1$}

\put (57.3,19) {$l$}

\put (16.5,15.6) {\line (1,0) {6}}

\put (25, 16) {\line (1, 0) {4}}

\put (44,16) {\line (1,0) {4}}

\put (50.8, 15.6) {\line (1,0){6}}

\put (50.8, 16.4) {\line (1,0){6}}

\end{picture}
\end{center}

\vskip -15mm
\noindent Let $\rho$ be the irreducible representation associated with the last fundamental dominant weight of
type $C_l$. Then ${\bf MSp_n} = \overline {K^*\rho (Sp_n)}$ is a semisimple monoid, and is called a symplectic
monoid. More details about symplectic monoids can be found in \cite{LZH, LR}. Let $MSp_n$ be the monoid consisting of fixed points of ${\bf MSp_n}$ under the Frobinus map of ${\bf MSp_n}$ to itself. The $MSp_n$ is referred to as a finite symplectic monoid.

\begin{theorem}\label{theorem OSP} Let $MSp_n$ be a finite symplectic monoid. Then
$$
|MSp_n| - 1 = (q-1) \sum_{r=0}^{l} ~q^{r^2}
                                       [l, r, q^2]^2
                                        \prod\limits_{i=1}^{r} (q^{2i} - 1)
                                        \prod\limits_{i=1}^{l-r} (q^{i} + 1)^2.
$$
\end{theorem}

\begin{proof}

\vspace{2mm}

It follows from Section 6 of \cite{PR2} that
\[
    \lambda^*(\Lambda \setminus \{0\}) = \{\phi, \{\a_l\}, \{\a_{l-1}, \a_l\},\dots, \{\a_1, \dots,\a_l\} \}.
\]
If $\lambda^*(e) = \emptyset$, then $\lambda(e) = \lambda_*(e) = \{\a_1, \dots, \a_{l-1}\}$ is of type $A_{l-1}$, and
has only one connected component with $|\lambda(e)| = l-1$. Thus, $ y(e) = (q-1)^{-1}\prod\limits_{i=1}^{l} (q^{2i} - 1)^2$ and $ z(e) = \prod\limits_{i=1}^{l-1} (q^{i+1} - 1)^2 $. Hence
\[
\frac {y(e)}{z(e)} = \frac { (q-1)\prod\limits_{i=1}^{l} (q^{2i} - 1)^2 } { \prod\limits_{i=1}^{l} (q^{i} - 1)^2 } = (q-1) \prod\limits_{i=1}^{l} (q^{i} + 1)^2.
\]

If $\lambda^*(e) = \{\a_{l-r+1}, \dots, \a_l \}$ with $1\le r\le l$,
then
\[
\lambda_*(e) =  \left\{ \begin{array}{ll}
                \{ \a_{1}, \dots, \a_{l-r-1} \} & \mbox{ if $r = 1, ..., l-2$}; \\
                \emptyset                        & \mbox{ if $r = l-1$ or $r=l$}.
                \end{array} \right.
\]
Thus, $\lambda^*(e)$ is of type $B_r$ and $\lambda_*(e)$ is of type $A_{l-r-1}$, where we agree that $A_{-1} = A_0 \cong 1$. It is clear that $|\lambda(e)|=l-1$. Thus,
$
    y(e) = ~q^{r^2} (q-1)^{-1} \prod\limits_{i=1}^{l} (q^{2i} - 1)^2
$
and
$
    z(e) = \prod \limits_{i=1}^{l-r-1} (q^{i+1}-1)^2 \prod\limits_{i=1}^{r} (q^{2i} - 1),
$
where $\prod\limits_{i=1}^{m} (q^{i+1}-1)^2 = 1$ if $m \le 0$. So
\[
\frac {y(e)}{z(e)} = \frac  { ~(q-1)q^{r^2}\prod\limits_{i=1}^{l} (q^{2i} - 1)^2 }
                            {\prod\limits_{i=1}^{l-r} (q^{i} - 1)^2\prod\limits_{i=1}^{r} (q^{2i} - 1) }.
\]
The formula follows from Theorem \ref{theorem OJS}      $\hfill\Box$

\end{proof}

All maximal chains in the cross section lattice $\Lambda$ of ${MSp_n}$ with $n = 2l$ have the same length. This yields a rank function on ${MSp_n}$ via the decomposition of
$$
    {MSp_n} = \bigsqcup_{x \in \Lambda} G x G
$$
where $G$ is the unit group of ${MSp_n}$. Note that this definition of rank is in general different from that of the rank of matrices. Denote by $M^r$ the set of elements in $MSp_n$ of rank $r = 0, 1, ..., l+1$. Clearly, $|M^0| = 1$. The proof of Theorem \ref{theorem OSP} indicates that
\[
    |M^r| = (q-1)q^{(r-1)^2} [l, r-1, q^2]^2
                                        ~\prod\limits_{i=1}^{r-1} (q^{2i} - 1)
                                        \prod\limits_{i=1}^{l-r+1} (q^{i} + 1)^2.
\]

\subsection{H-polynomials of Symplectic Monoids}

Let ${\bold B}$ be a Borel subgroup of the unit group of a symplectic monoid ${\bf MSp_n}$. Let ${\bold T \subseteq {\bold B}}$ denote a maximal torus, and $R$ denote the Renner monoid of ${\bold M}$. It follows from Definition 2.3 of Renner \cite{R8} that the $H$-polynomial of ${\bf MSp_n}$ is given by
\[
    H(t) = \sum_{R\setminus \{0\}} (t-1)^{r(x)-1}t^{l(x)-r(x)},
\]
where $r(x) = \dim({\bold T}x)$ and $l(x) = \dim({\bold B}x{\bold B})$. Renner \cite{R8} gave a systematic description of H-polynomials of semisimple monoids.

\begin{proposition}\label{H(q)} The H-polynomial of a symplectic monoid ${\bf MSp_n}$ is given by
\[
   H(q) =  \sum_{r=0}^{l} ~q^{r^2} [l, r, q^2]^2
                                        \prod\limits_{i=1}^{r} (q^{2i} - 1)
                                        \prod\limits_{i=1}^{l-r} (q^{i} + 1)^2.
\]
\end{proposition}

\begin{proof}
By both Proposition 4.1 and Theorem 4.1 of Renner \cite {R4} we have
\[
    |MSp_n| - 1 = (q-1) H(q).
\]
The result of this proposition follows from Theorem \ref{theorem OSP} $\hfill\Box$
\end{proof}

{\bf Remark} Can and Renner \cite{C2} obtained the $H$-polynomial of the symplectic rook monoid in a different way.

If $l=2$, by Proposition \ref{H(q)} we have
\[
    H(q) = 1 + q + q^2 + 2q^3 + 2q^4 + 2q^5+ 2q^6+ 2q^7 + q^8+ q^9+ q^{10}.
\]
Setting $k=0$ and $N=4$ in the notation of Example 5.2 of \cite{R6} and Example 6.2 of \cite {R8}, we see that the above $H(q)$ and the $H$-polynomial of type $C_2$ in Renner \cite{R6, R8} are the same.

If $l=3$ in Proposition \ref {H(q)}, then an elementary calculation yields that
\[
\begin{aligned}
    H(q) =   & 1 + q + q^2 + 2 q^3 + 2 q^4 + 3 q^5 + 4 q^6 + 4 q^7 + 4 q^8 + 5 q^9 + 5 q^{10} + 5 q^{11} \\
                        &+ 5 q^{12} + 4 q^{13} + 4 q^{14} + 4 q^{15} + 3 q^{16} + 2 q^{17} + 2 q^{18} + q^{19} + q^{20} +  q^{21}.
\end{aligned}
\]

Notice that the coefficients of $H(q)$ are Betti numbers of some topological spaces. This is not an accident. A general result is given below.

\begin{corollary} Let ${\bf M}$ be a symplectic monoid over the field $\mathbb C$ of complex numbers. Denote by $P({\bf M}) = ({\bf M} \setminus \{0\}) / {\mathbb C}^*$ the symplectic projective space, where ${\mathbb C}^* = {\mathbb C} \setminus \{0\}$. Then the sequence of coefficients of the following polynomial
\[
   H(q) =  \sum_{r=0}^{l} ~q^{r^2}[n, r, q^2]^2
                                        \prod\limits_{i=1}^{r} (q^{2i} - 1)
                                        \prod\limits_{i=1}^{l-r} (q^{i} + 1)^2
\]
consists of Betti numbers of $P({\bf M})$, and hence palindromic.
\end{corollary}

\begin{proof} Note that the polynomial $H(q)$ of ${\bf M}$ depends only on the Renner monoid of ${\bf M}$. On the other hand, ${\bf M}$ and the finite symplectic monoid $MSp_n$ have the same Renner monoid. Applying the result of Section 3 of Renner \cite{R6} and the above Proposition \ref{H(q)}, we see that this corollary is true.
\end{proof}

{\bf Open Question:}
How to determine whether a $\J$-irreducible monoids ${\bold M} = \overline {K^*\rho(G_0)}$ is $F_q$-split, where $G_0$ is a simple algebraic group over the algebraic closure of $F_q$ and $\rho: G_0\to GL(V)$ is an irreducible representation associated with any dominant weight?

\vspace{0.2cm} {\bf Acknowledgment: } We would like to thank Lex Renner for many useful suggestions, especially, for pointing out Theorem 2.1 and the sketch of its proof. We also thank Mohan Putcha and Reginald Koo for many helpful discussions.


\newpage

\vspace{10mm}
\noindent Zhuo Li\\
Department of Mathematics \\
Xiangtan University\\
Xiangtan, Hunan 411105, P. R. China\\
Email: zli@mail.xtu.edu.cn

\vspace{3mm} \noindent Zhenheng Li \\
Department of Mathematical Sciences \\
University of South Carolina Aiken\\
Aiken, SC 29801, USA\\
\noindent Email: zhenhengl@usca.edu

\vspace{3mm}
\noindent You'an Cao \\
Department of Mathematics \\
Xiangtan University\\
Xiangtan, Hunan 411105, P. R. China\\
Email: cya@mail.xtu.edu.cn

\end{document}